\documentclass[smallextended,envcountsect]{svjour3}
\usepackage{latexsym,amsmath,amsfonts,amscd,amssymb,verbatim,array,mathdots,mathtools,mleftright}
\usepackage{enumitem}
\usepackage{graphicx}
\RequirePackage{pgf}											
\RequirePackage{tikz}
\usetikzlibrary{
	arrows.meta,										
	decorations.text,									
	matrix,												%
	patterns,											
	patterns.meta,										%
	positioning,										%
	shapes,												
	through,                            
}
\RequirePackage{pgfplots}								
\pgfplotsset{compat=1.18}										
\usepgfplotslibrary{
	fillbetween,						
	decorations.softclip				
}
\usepackage{caption}
\pgfdeclarelayer{front}
\pgfdeclarelayer{back}
\pgfsetlayers{back,main,front}
\tikzset{
	on layer/.code={\pgfonlayer{#1}\tikzset{every picture}\begingroup\aftergroup\endpgfonlayer\aftergroup\endgroup},
	also in front/.style 2 args={#1,postaction={on layer=front,draw,#1,#2}},
	also behind/.style 2 args={#1,postaction={on layer=back,draw,#1,#2}},
}
\usepackage{tikz}
\usepackage[normalem]{ulem}

\newcommand{\Paul}[1]{{\color{green} #1}}

\smartqed

\usepackage[hidelinks]{hyperref}
\usepackage{cleveref}
\usepackage{chngcntr}
\counterwithin{equation}{section}

\DeclareRobustCommand{\norm}[2][]{			
	\begingroup
	\protect
	\if\relax\detokenize{#1}\relax			
	\ensuremath{\left\lVert#2\right\rVert}
	\else									
	\IfSubStr{bigBigbiggBigg}{#1}{
		\ensuremath{\#1\lVert#2\#1\rVert}
	}{
		\PackageWarning{Diese Größe ist mir nicht bekannt.}
	}
	\fi
	\endgroup
}

\DeclareMathOperator{\graph}{gph}
\DeclareMathOperator{\dom}{dom}
\DeclareMathOperator{\Sol}{Sol}

\DeclareMathOperator{\Solw}{Sol^w}

\DeclareMathOperator{\VOP}{VOP}
\DeclareMathOperator{\OP}{OP}

\DeclareMathOperator{\ri}{ri}
\DeclareMathOperator{\cl}{cl}
\DeclareMathOperator{\N}{\mathbb{N}}
\DeclareMathOperator{\R}{\mathbb{R}}

\DeclareMathOperator{\A}{\mathcal{A}}
\DeclareMathOperator{\B}{\mathcal{B}}

\DeclareMathOperator{\inte}{int}

\begin{document}
	\title{Unboundedness of the images of set-valued mappings having
		closed
		graphs: Application to vector optimization}
	\titlerunning{Unboundedness of the images of set-valued mappings}
	\author{V.T. Hieu \and E.A.S. K{\"o}bis \and M.A. K{\"o}bis \and P.H. Schm{\"o}lling}

	\institute{
 Vu Trung Hieu \\
Department of Mathematical Sciences, Norwegian University of Science and Technology \\ 
7491, Trondheim, Norway \\ 
Center for Advanced Intelligence Project, RIKEN \\
103-0027, Tokyo, Japan \\ 
\\
Elisabeth Anna Sophia K{\"o}bis \and Markus Arthur K{\"o}bis \and Paul Hugo Schm{\"o}lling \\
		Department of Mathematical Sciences, Norwegian University of Science and
		Technology 
		\\ 7491, Trondheim, Norway}


	\date{Received: date / Accepted: date}

	\maketitle

	\medskip

	\begin{abstract}
		In this paper, we propose criteria for unboundedness of the
		images of set-valued mappings having closed graphs in Euclidean spaces. We focus on mappings whose
		domains are non-closed or whose values are connected.
		These criteria allow us to see structural properties of solutions in vector
		optimization, where solution sets can
		be considered as the images of solution mappings associated to specific scalarization
		methods. In particular, we prove that if the domain of a certain solution mapping is non-closed, then the weak Pareto solution set
		is unbounded. Furthermore, for a quasi-convex problem, we demonstrate two criteria to ensure that if the weak Pareto solution set is disconnected then each connected component is unbounded.
	\end{abstract}

	\keywords{Set-valued mapping \and closed graph  \and  unboundedness \and disconnectedness \and vector
		optimization \and scalarization \and  quasi-convex function}

	\subclass{90C33 \and 90C31\and 54C60}

	\section{Introduction}

	Studying the structure of solution sets
	plays
	an important role in vector optimization theory; see, for example \cite[Chapter~6]{lucbook}
	and \cite[Chapter~3]{ehrgott2005multicriteria}.
	The most effective method for such studies is to consider the solution sets in question as the images of the solution mappings obtained by some scalarization techniques.

	It is well-known that the image of a connected set under a continuous single-valued mapping is connected. 
	Naccache in \cite{naccache1978connectedness},
	Warburton in \cite{warburton1983quasiconcave}, and  Hiriart-Urruty
	in \cite{hiriart1985images} extended the result for a set-valued mapping
	having connected values, provided that the mapping is
	upper (or lower) semi-continuous; see \Cref{thm:warb}. This criterion opened a
	door for studying connectedness of solution sets of not only vector optimization problems but also
	set optimization problems, vector variational
	inequalities and vector equilibrium problems. There are numerous works related to this
	topic, we can only list a few of them, for example
	\cite{luc1987connectedness,sun1996connectedness,gong2007connectedness,yen2011monotone,han2018existence,sharma2023connectedness,peng2019connectedness,han2022connectedness,anh2022connectedness}.
	Almost all therein obtained results require generalized
	convexity or monotonicity assumptions to ensure connected-valuedness, and require
	bound\-edness of the constraint sets or certain coercivity conditions to guarantee
	semi-continuity of solution mappings.

	The aim of this paper is to look for tools that can be applied to investigate
	unboundedness and disconnectedness of solutions to problems governed by unbounded data
	in optimization theory.

	In particular, the first part is dedicated to studying unbound\-edness of
	the images of set-valued mappings having closed graphs.  In
	\Cref{thm:main_sa},  we prove that if a
	mapping has non-closed domain then its image is unbounded. We next
	study the relationship between unboundedness and disconnectedness of the image of
	a mapping having connected values. In \Cref{thm:component},
	we introduce two criteria to ensure unboundedness of all connected components of
	the image.

	In the second part, we apply the above results to vector optimization based on two solution
	mappings associated to two scalarization methods.
	We obtain new properties related to
	unboundedness of the (weak) Pareto solution set of a problem in both
	cases: with a generalized
	convexity assumption and without imposing generalized convexity. For
	the second case, in \Cref{thm:VOP_nonconvex} we prove that if the domain of one of the two aforementioned
	solution mappings (denoted by $\Theta_1$ and $\Theta_2$; see the definitions in Subsection \ref{subs:scalar})
	is non-closed then the weak Pareto solution set is unbounded. The reverse implication holds for $\Theta_1$ if the problem is convex and this solution map has bounded values, see Theorem \ref{thm:VOPconvex}. In Theorem \ref{thm:VOPquasiconvex}, we have the same conclusion for $\Theta_2$ when the objective functions are quasi-convex and non-negative over the constraint set.

	For a problem governed by semi-strictly quasi-convex functions, to ensure
	disconnectedness of all components of
	the weak Pareto solution set,
	Hoa et al.  \cite[Theorem 3.2]{hoa2007unbounded} used the condition
	that the number of components of the set is finite.
	In \Cref{thm:semistrict}, we introduce two other criteria for this property
	of quasi-convex problems. As a
	result, when the objective functions are strictly
	quasi-convex, the Pareto solution set coincides with the weak
	Pareto solution set, and if this set is disconnected then all the
	components are unbounded.

	The remaining part of this paper consists of five sections. \Cref{sec_pre}
	gives some definitions, notations, and auxiliary results on set-valued mappings and
	connected sets. \Cref{sec:non_closed,sec:con} establish results concerning
	unboundedness of the images
	of set-valued mappings having non-closed domains and having connected values,
	respectively.
	Results on unboundedness of solutions to vector optimization
	problems are shown in the last two sections, namely \Cref{sec_dis,sec:convex}.

	\section{Preliminaries}\label{sec_pre}

	Let $X$ be a nonempty subset of $ \R^m$. 
	We consider a set-valued mapping $S$ from $X$
	to $\R^n$, $S:X\rightrightarrows\R^n$. The \textit{domain} of $S$ is the
	set of all
	elements $u$ in $X$ such that $S(u)$
	is non-empty:
	\begin{center}
		$\dom S\coloneqq\{u\in X:S(u)\neq \emptyset\}$.
	\end{center}
	The union of the images (or values) $S(u)$,
	when
	$u$ ranges over $U\subseteq X$, is the image of $U$ under the mapping $S$:
	$$S(U)\coloneqq\cup_{u\in U} S(u).$$
	The \textit{image} of $S$ is the set $S(X)=S(\dom S)$.
	If its \textit{graph}
	\begin{center}
		$\graph S\coloneqq\{(u,v)\in X\times \R^n:v\in S(u)\}$
	\end{center}
	is a closed subset of $X\times \R^n$, then one says that $S$ has \textit{closed
		graph}. If $S$ has closed graph, then $S(u)$ is a closed set for every $u\in X$.
	The \textit{inverse mapping} of $S$, denoted by $S^{-1}: \R^n
	\rightrightarrows  X$, is defined
	by
	$$S^{-1}(v)\coloneqq\{u\in X: v\in S(u)\}, \ v\in \R^n.$$
	If $\A$ is a subset in $\R^n$, then the \textit{inverse image} of $\A$ by $S$ is
	the following set
	\begin{center}
		$S^{-1}(\A)=\{u\in X:S(u)\cap \A\neq \emptyset\}.$
	\end{center}

	The mapping $S$ is
	\textit{locally bounded} at $\bar u$ if there exists an open neighborhood $U$
	of $\bar u$ such that $S(U\cap X)$ is bounded.
	If $S$ is locally bounded at every $u\in X$, then
	$S$ is
	said to be locally bounded on $X$.
	The  mapping $S$ is  \textit{upper semi-continuous} at $\bar u\in X$
	if for
	any open set $V\subseteq \R^n$ such that $S(\bar u)\subseteq V$ there exists an open
	neighborhood $U$ of $\bar u$ such that $S(u)\subseteq V$ for all $u\in U$.
	If $S$ is upper semi-continuous at every $u\in X$, then $S$ is
	said to be upper semi-continuous on $X$.

	From \cite[Theorem 1.4.1]{AuslenderTeboulle2003AsymptoticConesFunctionsOptimizationVariationalInequalitiesbook}, one has the following lemma which
	will be used in the proof of \Cref{thm:component}.

	\begin{lemma}\label{lm:local}
		Assume that $S$ has closed graph. If $S$ is locally bounded at $\bar u$ for
		some
		$\bar u$ in  $\dom S$, then
		$S$ is upper semi-continuous at $\bar u$.
	\end{lemma}

	Let $Y$ be a nonempty subset of $\R^n$.
	One says that $Y$ is \textit{connected} if there do not exist two nonempty disjoint subsets $A,B$ of $Y$ with $Y\subseteq A\cup B$ and two open subsets $U,V$
	in $\R^n$ such that $A\subseteq U$, $B\subseteq V$ and $(U\cap Y)\cap (V\cap Y)=\emptyset$.
	Otherwise $Y$ is called \textit{disconnected}.
	A nonempty subset $A\subseteq Y$ is said to be a \textit{connected component} of
	$Y$ if $A$ is connected and it is not a proper subset of any connected subset of
	$Y$. Recall that when $Y$ is connected, if the set $A\subseteq Y$ is closed and open
	in $Y$, then $A=Y$. Moreover, if $Y$ is connected then the union $Y\cup \{y\}$, where
	$y$ belongs to the closure $\cl(Y)$, is also connected.

	One says that the set-valued mapping $S$ has \textit{connected} (\textit{convex, bounded}) \textit{values} if
	$S(u)$ is connected (convex, bounded) for every $u$ in the domain of $S$.

	For the convenience of the readers in comparing our developments with the before-mentioned
	result by Naccache, Warburton, and Hiriart-Urruty, we recall the following theorem; its proofs can be
	found in \cite[Lemma
	2.1]{naccache1978connectedness},
	\cite[Theorem
	3.1]{warburton1983quasiconcave}, or \cite[Theorem
	3.1]{hiriart1985images}.

	\begin{theorem}\label{thm:warb}
		Assume that $\dom S$ is connected and $S$ has 
		connected values. If $S$ is upper
		semi-continuous on $X$, then the image $S(X)$ is connected.
	\end{theorem}



	\section{Set-valued mappings having non-closed domains}\label{sec:non_closed}

	Let $A$ be a \textit{nonempty} subset of $\R^m$.  We denote by $\ri A$ its \textit{relative interior}.
	Throughout the present section, it is assumed that the following set-valued mapping has \textit{closed graph}:
	$$\Phi:A\rightrightarrows\R^n, \ a \mapsto \Phi(a).$$

	The proofs of the main results in the present and next sections are based on the following
	lemma.

	\begin{lemma}\label{A_closed}
		If $\A$ is a bounded connected component of $\Phi(A)$, then
		$\Phi^{-1}(\A)$ is closed. Furthermore, if $\Phi(A)$ is bounded, then
		$\dom\Phi$ is closed.
	\end{lemma}
	\begin{proof}
		Let $\A$ be a bounded connected component of $\Phi(A)$. To prove that $\Phi^{-1}(\A)$ is closed, we suppose that a sequence $\{a^k\}\subseteq \Phi^{-1}(\A)$
		converges to $\bar a\in\R^m$ and need to prove $\bar a\in \Phi^{-1}(\A)$. Let $\{x^k\}\subseteq \A$ be a sequence such that $x^k\in \Phi(a^k)$ for every $k\in\N$. By the boundedness of $\A$,  we
		can assume that  $x^{k} \to \bar x$. Since the
		graph of $\Phi$ is closed, $(\bar a,\bar x)$
		belongs to $\graph \Phi$, thus $\bar x\in \Phi(\bar a)$. As
		$\A$ is connected, the union $\A\cup\{\bar x\}$ is
		connected. This means that $\bar x\in \A$. We thus get $\bar x\in  \A\cap
		\Phi(\bar a)$, and then $\bar a\in \Phi^{-1}(\A)$ by definition.
		Therefore, the inverse image $\Phi^{-1}(\A)$ is closed.

		We now prove the second conclusion. Suppose that $\Phi(A)$ is bounded. To prove closedness of $\dom\Phi$, we suppose that a sequence $\{a^k\}\subseteq \dom\Phi$ converges to $\bar a\in\R^m$ and need to prove $\bar a\in \dom\Phi$. Let $\{x^k\}\subseteq \Phi(A)$ be a sequence such that $x^k\in \Phi(a^k)$ for every $k\in\N$. By the boundedness of $\Phi(A)$,  we
		can assume that  $x^{k} \to \bar x$. Since the
		graph of $\Phi$ is closed, $(\bar a,\bar x)$
		belongs to $\graph \Phi$, thus $\bar x\in \Phi(\bar a)$. This implies that $\bar a$ belongs to $\dom\Phi$.
		\qed
	\end{proof}

	\begin{theorem}\label{thm:main_sa} Assume that $\dom\Phi$ is non-closed. Then,
		$\Phi(A)$ is unbounded. Moreover, if additionally $(\ri A)\cap\dom\Phi$ is dense in $\dom\Phi$, then $\Phi(\ri A)$ is unbounded.
	\end{theorem}
	\begin{proof}
		By the non-closedness of  $\dom\Phi$ and the second conclusion in \Cref{A_closed}, we conclude straightforwardly that $\Phi(A)$ is unbounded.  We
		only need to prove the second claim.

Let $(\ri A)\cap\dom\Phi$ be dense in $\dom\Phi$. Suppose on
		the contrary that $\Phi(\ri A)$ is bounded. Let
		$\overline \Sigma$ be the closure of the following set:
		\begin{center}
			$\Sigma\coloneqq\left(\graph\Phi\right) \cap\left(\ri A\times \R^n\right) =\{(a,x)\in \ri
			A\times
			\R^n:x\in\Phi(a)\}$.
		\end{center}
		Clearly, one has $\Sigma \subseteq \graph\Phi$.
		Since $\graph\Phi$ is closed, the closure of $\Sigma$ is also a subset of $\graph\Phi$, i.e. $\overline\Sigma\subseteq\graph\Phi$.
		Let $\pi:\R^{m+n}\to \R^{n}$ be the
		projection
		defined by
		$\pi(t_1,\dots,t_m,t_{m+1},\dots,t_{m+n})=(t_{m+1},\dots,t_{m+n})$.
It is not difficult to see that
		$$\Phi(\ri A) \subseteq \pi(\overline\Sigma) \subseteq \overline{\pi(\Sigma)}.$$
Moreover, by the boundedness of $\Phi(\ri A)$ and the definition of $\Sigma$, we see that $\pi(\Sigma)$ is bounded, hence $\overline{\pi(\Sigma)}$ is as well. Thus, we claim that the image of $\overline\Sigma$ under $\pi$ is bounded.

We now consider a new mapping $\Phi'$ defined by
		\[\Phi':A\rightrightarrows \R^n, \ \Phi'(a)\coloneqq\pi(\overline\Sigma)\cap\Phi(a).\]
		One can check that the graph of $\Phi'$ coincides with
		$\overline\Sigma$. Hence, $\Phi'$ also has closed graph. By boundedness of $\pi(\overline\Sigma)$, $\Phi'(A)=\pi(\overline\Sigma)$ is also bounded. Applying the second conclusion in
		\Cref{A_closed} to $\Phi'$, we get that $\dom\Phi'$ is closed.
		Furthermore, it is not difficult to see the following
		inclusions:
		$$(\ri A) \cap \dom\Phi\ \subseteq \ \dom\Phi' \ \subseteq  \ \dom\Phi.$$
By the density of $(\ri A) \cap \dom\Phi$ in $\dom\Phi$, the above inclusions imply that $\dom\Phi'$ is also dense in $\dom\Phi$.
		It follows that $\dom\Phi$ is closed. This is a contradiction.   \qed
	\end{proof}

	\begin{remark}\label{rmk:A}
		Let $\A$ be a connected
		component of $\Phi(A)$. If $\Phi^{-1}(\A)$ is
		non-closed, thanks to \Cref{A_closed}, $\A$ is
		unbounded.
	\end{remark}

	\begin{figure}[h!]
		\centering
		\begin{tikzpicture}[>=latex]
			\begin{axis}[
				axis x line=bottom,
				axis y line=center,
				xtick={-5,-4,...,5},
				ytick distance=10,
				xlabel={$x_1$},
				ylabel={$x_2$},
				every axis x label/.style={at={(current axis.right of origin)},anchor=west},
				ylabel style={above left},
				xmin=-5.5,
				xmax=5.5,
				ymin=0,
				ymax=53
				]
				\addplot [mark=none,thick,domain=-5.5:-0.1,samples=200] {1/x^2};
				\addplot [mark=none,thick,domain=0.1:5.5,samples=200] {1/x^2};
			\end{axis}
		\end{tikzpicture}
		\caption{The image of $\Theta_1$ has two unbounded components.}
	\end{figure}

	\begin{example}\label{ex:dis1}
		From \Cref{ex:nonconvex} we consider the set-valued mapping that is the solution
		mapping $\Theta_1:\Delta \rightrightarrows \R^2$ of the vector optimization problem given there:
		\begin{equation}\label{eq:ex1}
			\Theta_1(\xi_1)=\left\{\begin{array}{cl}
				\left\lbrace  \Big(\pm\sqrt{\frac{1-\xi_1}{\xi_1}},\frac{\xi_1}{1-\xi_1}\Big)\right\rbrace   & \
				\ \hbox{ if } \ \xi_1\in (0,1), \smallskip \\
				\emptyset & \ \ \hbox{ if } \ \xi_1\in\{0,1\} . \\
			\end{array}\right.
		\end{equation}
		We can see that the
		graph of $\Theta_1$ is
		closed and that the domain of $\Theta_1$ is the open interval $(0,1)$ which is
		non-closed.
		According to \Cref{thm:main_sa}, the image of $\Theta_1$
		is unbounded. In particular,
		we have
		\begin{equation}\label{eq:ex1_sol}
			\Theta_1([0,1]) = \big\lbrace (x_1,x_2)\in \R^2:
			x_1^2x_2=1\big\rbrace=\A \cup \B,
		\end{equation}
		where $$\A\coloneqq\Big\lbrace (x_1,x_2)\in \R^2: x_2>0,
		x_1\sqrt{x_2}=-1\Big\rbrace$$
		and $$\B\coloneqq\Big\lbrace (x_1,x_2)\in \R^2: x_2>0,
		x_1\sqrt{x_2}=1\Big\rbrace.$$
		Clearly, $\Theta_1^{-1}(\A)=\Theta_1^{-1}(\B)=\ri \Delta$ is non-closed. Hence,
		$\Theta_1([0,1])$ has two unbounded connected components $\A$ and
		$\B$, illustrated in Figure 1, whose inverse images are non-closed.
	\end{example}

	\section{Set-valued mappings having connected values}\label{sec:con}

	Throughout the section, we consider another set-valued mapping
	$$\Psi:A\rightrightarrows\R^n, \ a \mapsto \Psi(a),$$
	where $A\subseteq \R^m$ is nonempty, and
	$\Psi$ has \textit{closed graph} and \textit{connected
		values}.

	In \Cref{thm:warb}, under upper semi-continuity condition, the image of a
	connected set through $\Psi$ is connected. Consequently, if the image is disconnected
	then either $\Psi$ is not upper semi-continuous or the domain of $\Psi$ is
	disconnected.
	In this section, we focus on the relationship between disconnectedness of $\Psi(A)$
	and its unboundedness.

	\begin{lemma}\label{AB_disjoint}
		If $\A$ and $\B$ are different connected
		components of $\Psi(A)$, then
		\begin{equation}\label{empty}
			\Psi^{-1}(\A)\cap \Psi^{-1}(\B)=\emptyset.
		\end{equation}
	\end{lemma}
	\begin{proof}
		Suppose that $\A$ and $\B$ are different connected components of $\Psi(A)$. Assume on
		the contrary that there exists
		$a\in \Psi^{-1}(\A)\cap \Psi^{-1}(\B).$
		It follows that $\Psi(a)\cap \A\neq \emptyset$ and $\Psi(a)\cap \B\neq
		\emptyset$.
		Note that $\Psi(a)$ is connected. Since $\A,\B$ are connected components of $\Psi(A)$, one has $\Psi(a)\subseteq \A$ and $\Psi(a)\subseteq \B$. This
		yields
		$$\Psi(a)\subseteq \A\cap\B\neq\emptyset,$$
		which contradicts the assumption that $\A$ and $\B$ are two different connected components. Thus
		\eqref{empty}
		is proved. \qed
	\end{proof}

	The next two theorems provide criteria for unboundedness of all connected
	components of the image of a mapping having connected values.
	The technique used to prove these theorems is generalized from
	\cite{yen2011monotone,hieu2021disconnectedness}. The idea of these theorems is the
	following observation: we are interested in unboundedness, hence we ignore the
	unbounded values of $\Psi$ and focus on the points in $\dom \Psi$ having bounded
	values.  Here, we denote by $\mathcal{O}$ the set
	\begin{center}
		$\mathcal{O}=\{a \in \dom \Psi: \Psi(a) \text{ is bounded}\}$.
	\end{center}

	\begin{theorem}\label{thm:component}
		Assume that $\dom \Psi$ is connected and that $\Psi$
		is locally bounded or upper semi-continuous on $\mathcal{O}$.
		Then, if $\Psi(A)$ is disconnected, all its connected components are unbounded.
	\end{theorem}
	\begin{proof}
		Recall that the mapping $\Psi$ has closed graph. Thanks to \Cref{lm:local}, local boundedness of $\Psi$ implies its upper semi-continuity on
		$\mathcal{O}$. Hence, we only need to consider the case that $\Psi$ is upper
		semi-continuous on $\mathcal{O}$.

		Suppose that $\Psi(A)$ is disconnected and $\A$ is a connected
		component of $\Psi(A)$. On the contrary, we
		assume that  $\A$ is bounded. Let $\B$ be another connected component of $\Psi(A)$. Then, there exists an open
		set $V$ containing $\A$ such that $V\cap \B=\emptyset$. Recall from \Cref{AB_disjoint} that the two sets
		$\Psi^{-1}(\A)$ and $\Psi^{-1}(\B)$ are disjoint.

		By the closedness of the
		graph of $\Psi$ and \Cref{A_closed}, $\Psi^{-1}(\A)$ is closed.
		We now point out that
		$\Psi^{-1}(\A)$ is open in $\dom \Psi$. Indeed, let $a$ be an
		arbitrary point in $\Psi^{-1}(\A)$, then $\emptyset\neq\Psi(a)\subseteq \A$ is bounded.
		It follows that $a\in \mathcal{O}$.
		By
		assumption, $\Psi$ is upper semi-continuous at $a$. With the open set $V$ defined
		above, there is an open set $U$ in $\dom \Psi$ such that $\Psi(a')\subseteq V$ for all
		$a'\in U$.
		Hence, the openness of $\Psi^{-1}(\A)$ in $\dom\Psi$ is
		proved.

		From the above, $\Psi^{-1}(\A)$ is both open and closed in the connected set $\dom \Psi$.
		It follows that $\Psi^{-1}(\A)=\dom \Psi$, while $\dom \Psi$ contains $\Psi^{-1}(\B)\neq \emptyset$. This is impossible because of  $\Psi^{-1}(\A)\cap \Psi^{-1}(\B)=\emptyset$. Thus, $\A$ must be unbounded.
		\qed
	\end{proof}

	To illustrate \Cref{thm:component}, we consider an example coming from linear fractional vector optimization.

	\begin{example}
		Consider the solution  mapping of a linear fractional bi-criteria
		optimization problem in  \cite[Section 4]{hoa2005parametric}, which has the feasible set
		$$K=\{(x_1,x_2)\in\R^2: x_1\geq 0, x_2\geq 0, x_1+x_2\geq 1\}$$
		and the objective functions
		$$f_1(x)=\frac{x_1+1}{2x_1+x_2}, \ \ f_2(x)=\frac{-x_1-2}{x_1+x_2}.$$

		\begin{figure}[h!]
			\centering
			\begin{tikzpicture}[domain=0:4]

				\draw[->] (-0.2,0) -- (4.2,0) node[right] {$x_1$};
				\draw[->] (0,-0.2) -- (0,4.2) node[above] {$x_2$};

				\draw[thick] (1,0) -- (4.2,0) node[midway,above,scale=0.8] {$\xi_1=\frac12$};
				\draw[thick,domain=6.2/11.4:2/3,smooth,variable=\t] plot ({(2-3*\t)/(2*\t-1)},2);
				\draw[<-] (0.1,2.1) -- (1,2.4);
				\node[scale=0.8] at (2,2.4) {$\frac23>\xi_1>\frac12$};
				\draw[->] (3,2.4) -- (4,2.3);
				\draw[thick] (0,2) -- (0,4.2) node[midway,right,scale=0.8] {$\xi_1=\frac23$};
				\draw[<-] (1,0.1) -- (1,0.7) node[above right,scale=0.8] {$\xi_1\in[0,\frac12)$};
				\filldraw[black] (1,0) circle (1.5pt);
			\end{tikzpicture}
			\caption{The image of $\Theta_1$ has two unbounded components.}
		\end{figure}

		The solution mapping $\Theta_1$ of the problem (see \Cref{subs:scalar})
		can be written as follows:
		$$\Theta_1(\xi_1)=\left\{\begin{array}{cl}
			\{(1,0)\} & \ \ \hbox{ if } \ 0\leq \xi_1 < \frac{1}{2}, \smallskip \\

			[1,+\infty)\times\{0\} & \ \ \hbox{ if } \  \ \ \ \xi_1 =\frac{1}{2},
			\smallskip \\

			\left\lbrace \Big(\frac{2-3\xi_1}{2\xi_1-1},2\Big)\right\rbrace   & \ \ \hbox{ if } \ \frac{1}{2} < \xi_1 <
			\frac{2}{3}, \smallskip \\
			\{0\}\times [2,+\infty) & \ \ \hbox{ if } \ \ \ \ \xi_1 =\frac{2}{3},
			\smallskip \\
			\emptyset & \ \ \hbox{ if } \ \frac{2}{3} < \xi_1 \leq 1. \\
		\end{array}\right.$$
		The domain of $\Theta_1$ is the interval
		$[0,\frac{2}{3}]$ which is connected.
		Moreover, the set $\mathcal{O}$ is $[0,\frac{1}{2})\cup (\frac{1}{2},\frac{2}{3})$
		and $\Theta_1$ is single-valued and continuous on $\mathcal{O}$. Hence, $\Theta_1$
		satisfies the condition in \Cref{thm:component}. The image has two connected
		components,
		both of them are unbounded.
	\end{example}

	The below theorem proposes another criterion to replace the connectedness of $\dom \Psi$
	in \Cref{thm:component}.

	\begin{theorem}\label{thm:com2}
		Assume that $A$ is connected, $\mathcal{O}$ is open in
		$A$, and $\Psi$ is upper semi-continuous or locally bounded  on $\mathcal{O}$.
		Then, if $\Psi(A)$ is disconnected, all its connected components are unbounded.
	\end{theorem}
	\begin{proof}
		The proof is similar to the proof of \Cref{thm:component}. Note
		that the openness of $\Psi^{-1}(\A)$ in $\dom\Phi$ is
		proved under the assumptions that $\mathcal{O}$ is open in $A$
		and $A$ is connected.\qed
	\end{proof}

	The assumption related to the openness of $\mathcal{O}$ and local boundedness or upper
	semi-continuity of $\Psi$ on $\mathcal{O}$ in \Cref{thm:com2} seems heavy in general.
	However, it is satisfied for
	several cases such as solution mappings in strictly quasi-convex vector optimization
	(see \Cref{lm:strictly}), linear fractional vector
	optimization, or monotone affine/general vector inequalities
	\cite{yen2011monotone,hieu2021disconnectedness}.

	We consider the following example that comes from monotone affine vector inequalities in
	which the assumption in \Cref{thm:component} related to connectedness of the
	domain is not satisfied but the assumptions in \Cref{thm:com2} are fulfilled.

	\begin{example}\label{ex:dis2}
		We next consider the bi-criteria case of the monotone affine vector inequality
		given in the proof of  \cite[Theorem 3.2]{hieu2019numbers} whose
		feasible set is  $K=\R^2$, and the two mappings are
		$$F_1(x)=\begin{bmatrix}
			x_2+1 \\ -x_1+1

		\end{bmatrix}, \ F_2(x)=\begin{bmatrix}
			-x_2+1 \\ x_1+1
		\end{bmatrix}.$$
		The solution mapping $\Theta_1$ of the problem
		is as follows:
		$$\Theta_1(\xi_1)=\left\{\begin{array}{cl}
			\left\lbrace  \Big(\frac{1}{1-2\xi_1},\frac{-1}{1-2\xi_1}\Big)\right\rbrace  & \ \ \hbox{ if } \ 0 \leq \xi_1
			<\frac{1}{2} \ \hbox{ or } \ \frac{1}{2} < \xi_1
			\leq 1, \smallskip \\
			\emptyset & \ \ \hbox{ if } \qquad \qquad \ \, \xi_1 =\frac{1}{2}.\\
		\end{array}\right.$$
		Clearly, $\dom \Theta_1 = [0,\frac{1}{2})\cup (\frac{1}{2},1]=\mathcal{O}$ is not
		connected but open in $[0,1]$. One observers that $\Theta_1$ is
		single-valued and
		continuous
		on $\mathcal{O}$.
		Its image has two connected
		components, both are unbounded.
	\end{example}

	\begin{proposition}\label{prop:com3}
		Assume that $\Psi$ is upper semi-continuous on $\dom\Psi$ and that $\Psi^{-1}$ also has connected values. Then, $\Psi(A)$ is disconnected if and only if $\dom \Psi$ is so.
	\end{proposition}
	\begin{proof}
		Clearly, $\Psi^{-1}$ has closed graph and its image, namely $\dom
		\Psi\subseteq A$,
		is bounded. Thanks to \Cref{lm:local}, $\Psi^{-1}$ is upper semi-continuous on
		$\Psi(A)$. Hence, both $\Psi$  and $\Psi^{-1}$ are upper semi-continuous on their
		domains and have connected values. Therefore, the conclusion is implied from \Cref{thm:warb}.\qed
	\end{proof}

	\begin{example}
		We consider the mapping given in \Cref{ex:dis2}, which
		satisfies the condition in \Cref{prop:com3}.
		Clearly, the number of connected components of $\dom \Theta_1$ equals the one
		of the image $\Theta_1([0,1])$.
	\end{example}

	\section{Application to vector optimization without generalized convexity}\label{sec_dis}

	We study unboundedness of solution sets in vector optimization through
	solut\-ion mappings associated to specific scalarization methods.
	Let us recall solution concepts and fundamental results in vector
	optimization. 

	Let $K$ be a nonempty and closed 
	subset of $\R^n$ and
	$f=(f_1,\dots,f_m)$ be a continuous mapping from $K$ to $\R^m$.
	The vector minimization problem with the feasible set $K$, vector objective
	function $f$ and nonnegative orthant $\R^m_+$ as ordering cone is written formally as follows:
	\begin{equation*}
		\label{VOP} \VOP(K,f)\qquad {\rm Minimize}\ \; f(x)\ \ {\rm subject \ to}\ \  x\in K.
	\end{equation*}
	A point $x\in K$ is said to be  a \textit{Pareto solution} of $\VOP(K,f)$ if there exists no $y\in K$ such that
	$f(y)-f(x)$ belongs to $-\R^m_+\setminus\{0\}$.
	It is said to be  a {\it weak Pareto solution} of $\VOP(K,f)$ if there does not exist a $y\in K$
	such that $f(y)-f(x)$ belongs to $-\inte\R^m_+$. The Pareto solution set and weak Pareto
	solution set of this problem are respectively denoted by $\Sol(K,f)$ and $\Solw(K,f)$.


	\begin{remark}\label{rmk:Solw}
		Let $x$ be given in $K$. It is $x\in\Solw(K,f)$ if and only if for any $y\in K$ there exists $k\in\{1,\dots,m\}$ such that $f_k(y) \geq f_k(x)$. Indeed, $x\in\Solw(K,f)$ is equivalent to the condition that for every $y\in K$ it is $f(y)-f(x)\not\in-\inte\R^m_+$, which is the same as $f_k(y)\geq f_k(x)$ for some index $k$.
	\end{remark}

	If $f$ is a scalar function, so $f:K\to\R$, then we point out the fact, that in this case the problem $\VOP(K,f)$ is the well known scalar minimization problem, by writing it as $\OP(K,f)$. In this case the (weak) Pareto solutions are the minimizers of $f$ over $K$.

	Let $g: K \to \R$ be a function. The function $g$ is said to be \textit{bounded from below by $\gamma\in\R$} on $K$ if $g(x)\geq \gamma$ for all $x\in K$ and \textit{strictly bounded from below by $\gamma$} on $K$ if $g(x)>\gamma$ for all $x\in K$.

	\subsection{Two solution mappings}\label{subs:scalar}

	We now consider the $(m-1)$-standard simplex in $\R^m$
	\begin{equation*}\label{eq:Delta}
		\Delta\coloneqq\{(\xi_1,\dots,\xi_m)\in\R_+^m:\xi_1+\dots+\xi_m=1\}.
	\end{equation*}
	Then, its relative interior $\ri \Delta$ is the set of all $\xi$ in $\Delta$
	whose coordinates are positive. Clearly, $\Delta$ is nonempty, convex and compact.
	For each $\xi$ in $\Delta$, we define
	$$f_{\xi}\coloneqq\xi_1 f_1+\dots+\xi_m f_m.$$
	The solution mapping associated to this scalarization method of the vector
	optimization problem $\VOP(K,f)$ is defined and denoted as
	follows:
	\begin{equation}\label{eq:theta1}
		\Theta_1(\xi)\coloneqq\Sol(K,f_{\xi}),
	\end{equation}
	where $\Sol(K,f_{\xi})$ is the solution set
	of the scalar problem $\OP(K,f_{\xi})$.

	The mapping $\Theta_1$ is an effective tool to investigate the weak Pareto
	solution set and Pareto solution set of a vector optimization problem because of
	the following results which can be found in, for example, \cite[Proposition~3.9]{ehrgott2005multicriteria}
	or \cite[Corollaries~4.1, 4.2 \& 4.5]{ansari2018vector}:

	\begin{lemma}\label{lm:scalar}
		The following inclusions hold for the vector optimization problem $\VOP(K,f)$:
		\begin{equation}\label{eq:scalar}
			\Theta_1(\Delta) \subseteq \Solw(K,f) \ \text{and} \ \ \Theta_1(\ri\Delta) \subseteq
			\Sol(K,f).
		\end{equation}
		Furthermore, if $K$ is convex and $f$ is convex, then  $\Theta_1(\Delta)=\Solw(K,f)$.
	\end{lemma}

	\begin{remark}\label{lm:graph1}
		The solution mapping $\Theta_1$ has closed graph.
		Indeed, consider a convergent sequence $(\xi^k,x^k)$ in
		\begin{center}
			$\graph\Theta_1=\{(\xi,x)\in \Delta\times K:x\in\Theta_1(\xi)\}$
		\end{center}
		such that $(\xi^k,x^k)\to (\bar \xi,\bar x)$. By the closedness of $\Delta$ and
		$K$,
		$\bar \xi$ and $\bar x$ belong to $\Delta$ and $K$, respectively. Let $x\in K$
		be arbitrarily given. For
		each $k$, one
		has
		$f_{\xi^k}(x)\geq f_{\xi^k}(x^k)$. Take $k\to +\infty$, by the continuity of $f$,
		we
		get $f_{\bar\xi}(x)\geq f_{\bar\xi}(\bar x)$. It follows that $\bar x \in
		\Sol(K,f_{\bar\xi})$.
	\end{remark}

	We consider another solution mapping, which was introduced in
	\cite{warburton1983quasiconcave} and is also very useful for investigating solution
	sets of vector optimization problems governed by
	quasi-convex functions.
	For $\xi$ in $\Delta$, we define
	\begin{equation}\label{eq:fxi2}
		f^{\xi}(x) \coloneqq \max\left\lbrace \xi_1f_1(x), \ \dots \, , \
		\xi_mf_m(x)\right\rbrace.
	\end{equation}
	The solution mapping $\Theta_2:\Delta\rightrightarrows \R^n$ associated to this
	scalarization is given by
	\begin{equation}\label{eq:theta2}
		\Theta_2(\xi)=\Sol(K,f^{\xi}).
	\end{equation}

	\begin{remark}\label{rmk:Theta2_Solw}
		One has $\Theta_2(\ri\Delta)\subseteq\Solw(K,f)$. Indeed, for any $x\in\Theta_2(\ri\Delta)$, there exists $\xi\in\ri\Delta$ with $x\in\Sol(K,f^\xi)$. This implies that for every $y\in K$ there exists an index $k\in\{1,\dots,m\}$ such that for all $i\in\{1,\dots,m\}$ we have
		\[f^\xi(y)=\xi_kf_k(y)\geq\xi_if_i(x).\]
		For $i=k$ we therefore get $\xi_kf_k(y)\geq\xi_kf_k(x)$. Since $\xi\in\ri\Delta$ we have $\xi_k>0$, so that we arrive at $f_k(y)\geq f_k(x)$, which shows $x\in\Solw(K,f)$ (see \Cref{rmk:Solw}).

		Suppose that all objective functions $f_1,\dots,f_m$ are even strictly bounded from below by $0$ on $K$. Then, according to \cite[Theorem 3.4]{warburton1983quasiconcave}, we have the equality
		$\Theta_2(\ri\Delta)=\Solw(K,f)$.
		If additionally $\Theta_2$ is single-valued on its domain, then
		$\Sol(K,f)=\Solw(K,f)$, see \cite[Theorem 3.5]{warburton1983quasiconcave}.
	\end{remark}

	In the following theorem, we prove that, if the objective functions are bounded from below by $0$ on $K$, then the weak Pareto solution set is the image of $\Delta$ under $\Theta_2$.
	This theorem which is an improvement of  \cite[Theorem~3.4]{warburton1983quasiconcave} plays a vital role in the proof of Theorem \ref{thm:semistrict} in the next section.

	\begin{theorem}\label{lm:Theta2}
		If all objective functions $f_1,\dots,f_m$ are bounded from below by $0$ on $K$, then we have
		\begin{equation}\label{eq:SolwTheta2}
			\Solw(K,f)=\Theta_2(\Delta).
		\end{equation}
	\end{theorem}
	\begin{proof}
		For $m=1$ this is evident, so from now on we assume $m>1$.
		To prove \eqref{eq:SolwTheta2}, we first prove $\Solw(K,f)\subseteq\Theta_2(\Delta)$. For this, we consider any $x\in\Solw(K,f)$. If there exists an index $k\in\{1,\dots,m\}$ with $f_k(x)=0$, then $x\in\Sol(K,f^\xi)$, where $\xi=e^k$ is the $k$-th standard unit vector of $\R^m$. Indeed, for any $y\in K$ one has $f^\xi(y)=\max\{0,f_k(y)\}\geq0$. It is clear that $f^\xi(x)=0$, which implies  $x\in\Sol(K,f^\xi)$. From now on we assume $f(x)\in\inte\R^m_+$ and define $\xi\in\R^m$ by
		\begin{equation}
			\xi_i:=\frac{1/f_i(x)}{1/f_1(x) +\dots+ 1/f_m(x)}, \quad i=1,\dots,m. \label{gl:Theta2Xi}
		\end{equation}
		Then, it is $\xi\in\ri\Delta$ and we get $$f^\xi(x)=\frac1{1/f_1(x) +\dots+ 1/f_m(x)}.$$
		One has  $$\xi_if_i(x)=\frac1{1/f_1(x) +\dots+ 1/f_m(x)},$$
		for all $i\in\{1,\dots,m\}$. Because of $x\in\Solw(K,f)$, for any $y\in K$, there exists an index $k\in\{1,\dots,m\}$ with $f_k(y)\geq f_k(x)$ implicating
		\[
		\begin{aligned}
			f^\xi(y)\geq\xi_kf_k(y) & =\frac{f_k(y)/f_k(x)}{1/f_1(x) +\dots+ 1/f_m(x)}  \\
			&\geq\frac1{1/f_1(x) +\dots+ 1/f_m(x)}=f^\xi(x),
		\end{aligned}
		\]
		which shows $x\in\Sol(K,f^\xi)$.

		We now prove that $\Theta_2(\Delta)\subseteq\Solw(K,f)$.  \Cref{rmk:Theta2_Solw} points out that $\Theta_2(\ri\Delta)\subseteq\Solw(K,f)$. Thus, we only need to show $\Theta_2(\partial_r\Delta)\subseteq\Solw(K,f)$,

		where $\partial_r\Delta$ denotes the relative boundary $\partial_r\Delta=\overline\Delta\setminus\ri\Delta$.
		Consider $x\in\Theta_2(\partial_r\Delta)$. Again, there exists $\xi\in\partial_r\Delta$ such that $x\in\Sol(K,f^\xi)$. If there exists an index $k\in\{1,\dots,m\}$ with $f_k(x)=0$, then because of the lower boundedness we immediately get $x\in\Solw(K,f)$, so that from now on we can assume $f(x)\in\inte\R^m_+$. Define
		\[I(\xi)\coloneqq\{i\in\{1,\dots,m\}\mid\xi_i>0\}.\]
		Then, since for all $i\in\{1,\dots,m\}\setminus I(\xi)$, we have $\xi_if_i(x)=0$ it is
		\begin{equation}\label{eq:fI}
			f^\xi(x)=\max_{i\in I(\xi)}\xi_if_i(x)>0.
		\end{equation}
		Now, as before, for every $y\in K$ there exists an index $k\in\{1,\dots,m\}$ such that for all $i\in\{1,\dots,m\}$ it is $\xi_kf_k(y)\geq\xi_if_i(x)$ and from $f^\xi(x)>0$ we actually get $k\in I(\xi)$, which again allows to conclude $f_k(y)\geq f_k(x)$ and therefore, according to \Cref{rmk:Solw}, $x\in\Solw(K,f)$, which proves \eqref{eq:SolwTheta2}. \qed
	\end{proof}


	\begin{remark}
		In \cite[p.~543]{warburton1983quasiconcave} it is stated that the claim of \cite[Theorem 3.4]{warburton1983quasiconcave} remains correct for $f$ bounded from below by any value $\gamma$. However, we want to point out that 
		the bound zero is of utmost importance and cannot be generalized without further assumptions. In the following, we can see a counter-example in which  $\Solw(K,f)\neq\Theta_2(\Delta)$.
	\end{remark}

	\if
	Hieu: I would like to choose the first example to illustrate the last remark because it is the simplest, and all readers (even reviewers) can follow easily/quickly. Could you i) revise the example (don't need to use extra notations such as X, Y, WMin, id) \Paul{done, could you please check whether I used the right notations? For example, I use $\subseteq$ when you use $\subseteq$ and $\subseteq$ when you use ... $\subsetneq$? Also when refering to earlier lemmas, remarks, ... I write them lowercase, while you use uppercase - correct?. I tried to adapt, but I am quite bad at finding such things.}, ii) the picture: "dashed" is hard to look, the red color may not appear in the printed version

	Paul: I changed the picture a little. Do you have ideas about how to replace the colours? Or wishes where to put the labels? I would make it thicker and after that I'm lost...
	\fi

	\begin{example}
		Consider $f:\R^2\to\R^2$, $f(x)\coloneqq x$ and $K\coloneqq[-1,0]\times[0,1]$. It can easily be seen that
		\[\Solw(K,f)=\big([-1,0]\times\{0\}\big)\cup\big(\{-1\}\times[0,1]\big).\]
		For any $\xi\in\Delta$
		and $x\in K$, it is $\max\{\xi_1x_1,\xi_2x_2\}=\xi_2x_2$, because the first part of the maximum is always non-positive while the second part is always non-negative. Therefore, we have
		\[\Theta_2(\xi)=\begin{cases}
			[-1,0]\times\{0\} & \text{if}\ \xi_2\neq0 \\
			K & \text{if}\ \xi_2=0
		\end{cases}\]
		which gives
		\[\Theta_2(\ri\Delta)=[-1,0]\times\{0\}\subsetneq\Solw(K,f)\subsetneq K=\Theta_2(\partial_r\Delta)=\Theta_2(\Delta).\]
		To illustrate the sets, in \Cref{fig:enter-label}, $\Theta_2(\ri\Delta)$ is colored by blue and dashed, and $\Solw(K,f)$ is colored by red.
	\end{example}

	\begin{figure}
		\centering
		\begin{tikzpicture}[>=latex]
			\begin{axis}[
				clip=false,
				axis x line=center,
				axis y line=center,
				xtick={-1,0},
				ytick={0,1},
				xlabel={$x_1$},
				ylabel={$x_2$},
				every y tick label/.style={anchor=near yticklabel opposite,	xshift=0.2em,},
				ylabel style={above left},
				xmin=-1.3,
				xmax=0.3,
				ymin=-0.2,
				ymax=1.3
				]
				\let\tikzerror\relax
				\draw (0,0) -- (-1,0) -- (-1,1) -- (0,1) -- (0,0);
				\draw[pattern=dots] (-1,0) rectangle (0,1) node at (-0.3,0.7) {$\Theta_2(\Delta)$};
				\draw[also in front={red}{line width=.5\pgflinewidth},ultra thick] (0,0) -- (-1,0) -- (-1,1) node at (-0.7,0.3) {$\Solw(K,f)$};
				\draw[also behind={}{blue},also in front={blue,dashed}{line width=.5\pgflinewidth,dash phase=3pt},blue,ultra thick] (-1,0) -- (0,0) node[midway, below] {$\Theta_2(\ri\Delta)$};
			\end{axis}
		\end{tikzpicture}
		\hfill\caption{The relationship $\Theta_2(\ri\Delta) \varsubsetneq \Solw(K,f) \varsubsetneq \Theta_2(\Delta)$.}
		\label{fig:enter-label}
	\end{figure}

	\begin{remark}\label{lm:graph2}
		Similarly to $\Theta_1$,
		the solution mapping $\Theta_2$ also has closed graph. To demonstrate this remark,
		we
		consider a convergent sequence $(\xi^k,x^k)$ in
		\begin{center}
			$\graph\Theta_2=\{(\xi,x)\in \Delta\times K:x\in\Theta_2(\xi)\}$
		\end{center}
		such that $(\xi^k,x^k)\to (\bar \xi,\bar x)$. By the closedness of $\Delta$ and
		$K$,
		$\bar \xi$ and $\bar x$ belong to $\Delta$ and $K$, respectively. Let $x\in K$
		be arbitrarily given. For
		each $k$, one
		has
		$$\max\left\lbrace \xi^k_1f_1(x), \ \dots, \ \xi^k_mf_m(x)\right\rbrace  \ \geq \
		\max\left\lbrace \xi^k_1f_1(x^k), \
		\dots,
		\ \xi^k_mf_m(x^k)\right\rbrace .$$
		Take $k\to +\infty$. By the continuity of $f$,
		we get $f^{\bar\xi}(x)\geq f^{\bar\xi}(\bar x)$. It follows that $\bar x \in
		\Sol(K,f^{\bar\xi})$.
	\end{remark}

	\subsection{Problems without imposing any generalized convexity}

	We emphasize that, in the following theorem, we do not need any generalized convexity
	assumptions for the objective functions. By \Cref{lm:graph1,lm:graph2} the
	results obtained in \Cref{sec:non_closed}
	can be applied to both $\Theta_1$ and $\Theta_2$.
	We prefer to omit a detailed formulation of the results related to $\Theta_2$, which
	requires additionally that the objective functions are bounded from below by $0$.

	\begin{theorem}\label{thm:VOP_nonconvex}
		Assume that the domain of the solution
		mapping $\Theta_1$ given in \eqref{eq:theta1} is
		non-closed. Then, the weak Pareto solution set is unbounded. Moreover, the Pareto
		solution set is unbounded provided that $(\ri\Delta)\cap\dom\Theta_1$ is dense in $\dom\Theta_1$.
	\end{theorem}
	\begin{proof}
		According to \Cref{lm:graph1}, the solution
		mapping $\Theta_1$ has
		closed graph. The conclusions are obtained straightforwardly by applying the
		conclu\-sions from \Cref{thm:main_sa} 
		and the inclusions in \eqref{eq:scalar}. \qed
	\end{proof}

	\begin{remark}
		Let $\A$ be a connected
		component of $\Theta_1(\Delta)$, where $\Theta_1$ is given in \eqref{eq:theta1}. If
		$\Theta_1^{-1}(\A)$ is non-closed, then $\A$ is unbounded. Indeed, according to \Cref{lm:graph1}, $\Theta_1$ has
		closed graph. By applying straightforwardly
		\Cref{rmk:A}, we obtain the conclusion.
	\end{remark}

	\begin{example}\label{ex:nonconvex}
		Consider the bicriteria optimization problem $\VOP(K,f)$ in $\R^2$, with the
		feasible set $K=\R^2$ and two objective functions given by
		$$f_1(x)=x_1^4-2x_2, \quad f_2(x)=-2x_1^2+x_2^2.$$
		One can see that the functions are not bounded from below on $K$.
		Thus, the two end points $(0,1),(1,0)$ in $\Delta$ do not belong to the domain of the
		solution mapping $\Theta_1$ of $\VOP(K,f)$.
		Let $\xi\in\ri \Delta$, equivalently $\xi_1\in (0,1)$. One has
		$f_{\xi}=\xi_1 f_1+ (1-\xi_1)f_2$ and an easy computation shows that
		$$f_{\xi}(x_1,x_2)=\xi_1\left( x_1^2-\frac{1-\xi_1}{\xi_1}\right)^2+(1-\xi_1)
		\left(x_2-\frac{\xi_1}{1-\xi_1}\right)^2+h(\xi_1),$$
		where $h$ is a function depended only on $\xi_1$. Since the first two terms in
		$f_{\xi}$ are nonnegative, $f_{\xi}$ reaches its minimum if and only if these terms are zero.
		Hence, the solution
		mapping $\Theta_1$ of $\VOP(K,f)$ is as given in \eqref{eq:ex1} in \Cref{ex:dis1} and its domain
		is
		$$\dom\Theta_1=\Delta\setminus\{(0,1),(1,0)\}=\ri \Delta.$$
		This set is non-closed and
		its image is given as in \eqref{eq:ex1_sol}. Hence, according to \Cref{thm:VOP_nonconvex}, the
		Pareto solution set and weak Pareto solution set of the problem are unbounded as
		$\Sol(K,f) \supseteq \Theta_1([0,1])$.
	\end{example}

	\begin{remark}\label{rmk:nonclosed_Theta2} By repeating the argument in the proof of \Cref{thm:VOP_nonconvex}, we have the same conclusion for the mapping $\Theta_2$ given in \eqref{eq:theta2}. If the objective functions are bounded from below by $0$ and the domain of the solution
		mapping $\Theta_2$ given in \eqref{eq:theta2} is
		non-closed, then the weak Pareto solution set is unbounded.
	\end{remark}

	\section{Application to quasi-convex vector optimization}\label{sec:convex}

	We now recall some notions related to generalized convexity. We refer to the monographs
	\cite{ansari2013generalized,avriel2010generalized} for a general introduction to
	generalized convexity.

	Consider a function $g: K \to \R$.  This function is \textit{quasi-convex} on $K$
	if for $x, y$ in $K$, $x\neq y$, it is
	\begin{equation}\label{eq:qcx}
		g(tx+(1-t)y) \leq \max\{g(x),g(y)\},
	\end{equation}
	for all $t\in (0,1)$. Equivalently, $g$ is quasi-convex if and only if the level set
	$[g\leq \lambda]$
	is a convex set for all $\lambda \in \R$. Hence, the solution set of the
	optimization problem $\OP(K,g)$ is convex. In addition, when the inequality in
	\eqref{eq:qcx} is replaced by the
	strict one, we say that $g$ is \textit{strictly} quasi-convex on $K$. Recall that if
	$g$ is strictly quasi-convex, then the solution set of
	$\OP(K,g)$ has at most one element.

	\subsection{Stability of solutions in parametric optimization under quasi-convexity}

	Let $T$ be a convex subset of $\R^k$ and $g:K\times T\to \R$ be continuous on $K\times T$.
	We consider the parametric
	optimization problem $\OP(K,g(\cdot,t))$, where $t\in T$, and define the solution
	mapping $\Theta$ as follows:
	\begin{equation}\label{eq:theta}
		\Theta(t)\coloneqq\Sol(K,g(\cdot,t)).
	\end{equation}
	The following lemma about stability of solutions in quasi-convex optimization plays an important role in this section; its proof can be found in
	\cite[Theorem~4.3.3]{bank1983non}:

	\begin{lemma}\label{lm:usc_sol}
		Assume that, for any parameter $t\in T$, $g(\cdot,t)$ is quasi-convex on $K$.
		Then, the solution mapping
		$\Theta$ defined in \eqref{eq:theta} is upper semi-continuous at all points $t\in T$ for which
		$\Theta(t)$ is
		nonempty and bounded.
	\end{lemma}

	Considering the above parametric problem under the condition of strict quasi-convexity, we have
	the following lemma which will be used in the proof of \Cref{cor:semistrict}.

	\begin{lemma}\label{lm:strictly}
		Assume that, for any parameter $t\in T$, $g(\cdot,t)$ is strictly quasi-convex on $K$.
		Then, the solution map $\Theta$ is
		single-valued and continuous on its domain which is open in $T$.
	\end{lemma}
	\begin{proof}
		The single-valuedness of $\Theta$ is implied by the strict
		quasi-convexity of $g(\cdot,t)$ on $K$. Hence, for arbitrary $t\in \dom \Theta$,
		$\Theta(t)$ is the unique isolated local solution of $\OP(K,g(\cdot,t))$. Applying
		\cite[Theorem
		1]{klatte1985stability} to the considered parametric optimization problem, we
		conclude that there exists a bounded and open set $V\supseteq \Theta(t)$ and a
		neighborhood $U$ of $t$ such that, for all $t'\in U$, $\Theta(t')$ is nonempty and
		contained in $V$. Hence, the domain of $\Theta$ is open and $\Theta$  is
		continuous on
		this set.
		\qed
	\end{proof}

	The above lemmas will be used in the proofs of \Cref{thm:semistrict}, \Cref{thm:VOPquasiconvex}, and \Cref{thm:VOPconvex} where the parameter space $T$ is $\Delta$ and the solution mapping  $\Theta$ is $\Theta_1$ or $\Theta_2$. 

	\subsection{Problems governed by quasi-convex functions}


	To ensure unboundedness
	of all components of
	the weak Pareto solution set of a quasi-convex vector optimization problem,
	the authors in  \cite[Theorem~3.2]{hoa2007unbounded} used the condition that
	the number of components
	of the solution set is finite. In the following theorem, we propose two other conditions
	concerning the solution mapping $\Theta_2$ given in \eqref{eq:theta2} to
	replace it. Let
	\begin{equation}\label{eq:O}
		\mathcal{O}_2=\{\xi \in \dom \Theta_2: \Theta_2(\xi) \text{ is bounded}\}.
	\end{equation}

	\begin{theorem}\label{thm:semistrict}
		Let $f_1,\dots,f_m$ be quasi-convex and bounded from below by $0$ on $K$. Assume that one of the following conditions holds:
		\begin{description}
			\item[\rm(a)] $\dom \Theta_2$ is connected;
			\item[\rm(b)] $\mathcal{O}_2$ is open in $\Delta$.
		\end{description}
		If $\Solw(K,f)$ is
		disconnected, then all its components are
		unbounded. Conse\-quently, if the set is bounded then it is connected.
	\end{theorem}
	\begin{proof}
		Since the objective functions are bounded from below by $0$ on $K$, thanks to \Cref{lm:Theta2}, $\Solw(K,f)=\Theta_2(\Delta)$. Moreover, $f_1,\dots,f_m$ are quasi-convex, from \cite[Theorem 1.31]{ansari2013generalized}, $f^{\xi}$ given in \eqref{eq:fxi2} is also quasi-convex on $K$, for $\xi\in\Delta$, and, then $\Theta_2$ has convex values.

		According to \Cref{lm:usc_sol}, $\Theta_2$ is upper semi-continuous on
		$\mathcal{O}_2$.
		Therefore, the desired conclusion is obtained by applying directly
		Theorems  \ref{thm:component} and \ref{thm:com2} for $\Theta_2$ with the note that $\Theta_2$ has
		closed graph and connected values. \qed
	\end{proof}

	If the solution mapping $\Theta_2$ is single-valued on its domain, we have the following result:

	\begin{corollary}\label{cor:semistrict}
		Let all $f_i$ be quasi-convex and strictly bounded from below by $0$ on $K$.
		Suppose that $\Theta_2$ is single-valued on its domain. If $\Sol(K,f)$ is
		disconnected, then all its components are unbounded.
	\end{corollary}
	\begin{proof}
		From the assumption, the domain of $\Theta_2$ is $\mathcal{O}_2$, where
		$\mathcal{O}_2$ is defined by \eqref{eq:O}. Hence, $\mathcal{O}_2$ is the set of $\xi$ in $\Delta$ such that
		$\Theta_2(\xi)$ has only one element.
		Thanks to \cite[Theorem 3.5]{warburton1983quasiconcave}, we claim that
		$\Sol(K,f)$ coincides with $\Solw(K,f)$.

		By applying \Cref{lm:strictly} for the problem $\VOP(K,f)$, we conclude that
		$\dom\Theta_2$ is open in $\Delta$. As $\VOP(K,f)$ satisfies
		condition $\rm (b)$ in \Cref{thm:semistrict}, by applying that theorem for
		$\Solw(K,f)$, we obtain the desired conclusion.
		\qed
	\end{proof}

	\begin{example}
		We consider the problem given in \cite[Example~5.1]{warburton1983quasiconcave} whose constraint set is  $$K=\{x\in\R^2:2\leq x_1
		, 0\leq x_2\leq 4\}$$ and objective  functions are
		$$f_1(x)=\frac{x_1}{x_1+x_2-1}, \ \ f_2(x)=\frac{x_1}{x_1-x_2+3}.$$
		Both $f_1,f_2$ are strictly bounded from below by $0$ on $K$.
		This is a
		quasi-convex problem and its solution mapping $\Theta_2$ is given
		by
		$$\Theta_2(\xi_1)=\left\{\begin{array}{cl}
			\emptyset  & \ \ \hbox{ if } \ 0 \leq \xi_1
			<\frac{1}{2} \hbox{ or } \ \frac{5}{6} < \xi_1
			\leq 1, \smallskip \\
			\left\lbrace \Big(\frac{1+2\xi_1}{1-2\xi_1},0\Big)\right\rbrace  & \ \ \hbox{ if } \qquad \quad \ \frac{1}{6}
			\leq
			\xi_1
			<\frac{1}{2}, \smallskip \\
			\left\lbrace \Big(\frac{3-2\xi_1}{2\xi_1-1},4\Big)\right\rbrace  & \ \ \hbox{ if } \qquad \quad \ \frac{1}{2} <
			\xi_1 \leq \frac{5}{6}.
		\end{array}\right.$$
		Clearly, the domain of $\Theta_2$ is the union of two open intervals
		$\left(\frac{1}{6},\frac{1}{2} \right)\cup \left(\frac{1}{2},\frac{5}{6}
		\right)$; it is disconnected.
		The mapping is single-valued on its domain. According to
		\Cref{cor:semistrict}, the Pareto solution set and weak Pareto
		solution set coincide. They are the union of two disconnected parallel
		rays:
		$$\{(x_1,x_2)\in\R^2:x_1\geq 2, x_2=0\} \cup \{(x_1,x_2)\in\R^2:x_1\geq 2, x_2=4\}.$$
	\end{example}

	In the case that all objective functions are strictly quasi-convex, $\Theta_2$ is single-valued on its domain. Therefore, we have the following result:

	\begin{corollary}
		Suppose that all $f_i$ are strictly quasi-convex and strictly bounded from below by $0$ on
		$K$. Then, one has $\Sol(K,f)=\Solw(K,f)$. Moreover,  if $\Sol(K,f)$ is
		disconnected, then all its components are unbounded.
	\end{corollary}
	\begin{proof}
		For $\xi\in\Delta$, $f^{\xi}$ is also strictly quasi-convex on $K$ (see \cite[Lemma 4.2]{warburton1983quasiconcave});
		Hence $\Theta_2$ is single-valued on its domain. The conclusions are
		implied straight\-forwardly by \Cref{cor:semistrict}. \qed
	\end{proof}

	In the following theorem, we provide a necessary and sufficient condition for unboundedness of the weak Pareto solution set of $\VOP(K,f)$ under certain conditions:

	\begin{theorem}\label{thm:VOPquasiconvex}
		Let $f_1,\dots,f_m$ be quasi-convex and bounded from below by $0$ on $K$. Assume that $\Theta_2$ has bounded values.
		Then, the weak Pareto solution set is unbounded if and only if $\dom \Theta_2$ is non-closed in $\Delta$.
	\end{theorem}

	\begin{proof} According to \Cref{lm:Theta2}, one has $\Solw(K,f)=\Theta_2(\Delta)$. Clearly, if $\dom \Theta_2$ is non-closed in $\Delta$, thanks to \Cref{rmk:nonclosed_Theta2},
		the weak Pareto solution set is unbounded.

		We now prove the inverse conclusion. Assume that $\Solw(K,f)$ is unbounded, but on the contrary, $\dom \Theta_2$ is closed in $\Delta$.
		Let $\{x^k\}$ be an unbounded sequence in $\Solw(K,f)$. Because of $\Solw(K,f)=\Theta_2(\Delta)$, there exist $\{\xi^k\} \subset \dom \Theta_2  \subset \Delta$ such that $x^k\in\Sol(K,f_{\xi^k})$, for each $k$. By boundedness and closedness of $\dom \Theta_2$, we assume that $\xi^k \to \bar \xi \in \dom \Theta_2$, i.e., $\Theta_2(\bar \xi)\neq \emptyset$. Since $\Theta_2$ has bounded values, $\Theta_2(\bar \xi)$ is bounded.

		Let $O$ be an open and bounded ball containing $\Theta_2(\bar \xi)$. Thanks to \Cref{lm:usc_sol}, $\Theta_2$ is upper semi-continuous at $\bar \xi$, hence there exists an open neighbourhood $\Delta_O$ of $\bar \xi$ such that $\Theta_2(\xi)\subset O$ for all $\xi\in \Delta_O$. This implies that $x^k\in O$ for all $k$ large enough. This is impossible since $\{x^k\}$ is unbounded. \qed
	\end{proof}


	\subsection{Problems governed by convex functions}

	For problems given by convex data,  we also have a necessary and sufficient condition for unboundedness of the weak Pareto solution set of $\VOP(K,f)$ as follows:

	\begin{theorem}\label{thm:VOPconvex} Let $K$ be convex and $f_1,\dots,f_m$ be convex over $K$. Assume that $\Theta_1$ has bounded values. Then, the weak Pareto solution set is unbounded if and only if $\dom \Phi_1$ is non-closed in $\Delta$.
	\end{theorem}

	\begin{proof} We consider the scalarization $f_{\xi}=\xi_1 f_1+\dots+\xi_m f_m$, for $\xi\in\Delta$. Clearly, $f_{\xi}$ is also convex over $K$. Moreover, according to Lemma \ref{lm:scalar}, one has $\Solw(K,f)=\Phi_1(\Delta)$. The conclusion is proved by repeating the argument in the proof of Theorem \ref{thm:VOPquasiconvex}.  \qed
	\end{proof}

	\begin{remark}
		If $K$ is convex and all $f_i$ are convex over $K$, we have the same conclusions in Theorem \ref{thm:semistrict} for the solution map $\Theta_1$ defined by \eqref{eq:theta1}. However, until now, we still do not know any explicit convex vector optimization problems whose (weak) Pareto solution set is disconnected. This is related to Question 9.8 in \cite{yen2011monotone} and Question 5.1 in \cite{hieu2021disconnectedness}.
	\end{remark}

	When $\VOP(K,f)$ is given by convex polynomial data, i.e. $f_1,\dots,f_m$ are polynomials and $K$ is given by convex polynomials as follows:
	\[K=\{x\in\R^n: g_1(x)\leq 0, \dots, g_s(x)\leq 0 \},\]
	under a mild condition, in the following theorem, we claim that $\Solw(K,f)$ is connected.

	\begin{theorem}\label{thm:VOP_pol} Let $f_1,\dots,f_m$ and $g_1,\dots,g_s$ be polynomial and convex. Assume that $\Theta_1$ has bounded values. Then, the weak Pareto solution set is connected.
	\end{theorem}

	\begin{proof} Recall that $f_{\xi}$ is also convex over $K$ and $\Solw(K,f)=\Phi_1(\Delta)$ (see Lemma \ref{lm:scalar}). Assume that $\Solw(K,f)$ is nonempty, this implies that $\dom \Phi_1$ is so. We prove that $\dom \Phi_1$ is convex. Indeed, let $\xi', \xi''$ be two points in $\dom \Phi_1$, i.e., $f_{\xi'}$ and $f_{\xi''}$ are bounded from below over $K$. This implies that $f_{\xi}$ is also bounded from below over $K$ for all $\xi$ in the line segment connecting $\xi'$ and $\xi''$, denoted by $[\xi',\xi'']$. According to the Frank–Wolfe type theorem in convex polynomial optimization \cite{belousov2002frank}, $\Sol(K,f_{\xi})$ is nonempty for all $\xi\in [\xi',\xi'']$. Hence, $\dom \Phi_1$ is convex.

		Therefore, connectedness of $\Solw(K,f)$ is straightforward from Lemma \ref{lm:usc_sol} and Theorem \ref{thm:warb}. \qed
	\end{proof}

	To end this section, we illustrate Theorems \ref{thm:VOPconvex} and \ref{thm:VOP_pol} by considering the following example:

	\begin{example}\label{ex:convex}
		Consider the convex polynomial  bicriteria  optimization problem $\VOP(K,f)$ in $\R^2$, with the convex
		feasible set $K=\R_+^2$ and two convex objective functions given by
		$$f_1(x)=x_1^2-x_2, \quad f_2(x)=-x_1+x_2^2.$$
		Clearly, these functions are not bounded from below on $K$.
		Thus, the two end points $(0,1),(1,0)$ in $\Delta$ do not belong to the domain of the
		solution mapping $\Theta_1$ of $\VOP(K,f)$.
		Let $\xi\in\ri \Delta$, equivalently $\xi_1\in (0,1)$. One has
		$$f_{\xi}=\xi_1 f_1+ (1-\xi_1)f_2=\xi_1x_1^2+(1-\xi_1)x_2^2-(1-\xi_1)x_1-\xi_1x_2,$$
		which is quadratic and convex over $K$. An easy computation shows that
		$\Sol(K,f_{\xi})$ has the unique solution
		$$ \Sol(K,f_{\xi}) = \Big\{\Big(\frac{1-\xi_1}{2\xi_1},\frac{\xi_1}{2(1-\xi_1)}\Big)\Big\}, \, \xi_1\in(0,1).$$
		One sees that the weak Pareto solution set is connected, moreover, the set is unbounded and $\Theta_1$ has non-closed domain.
	\end{example}

	\section*{Perspectives and future work}

	The results related to set-valued mappings having connected values in this paper can be
	useful for investigating (un)bounedness and (dis)connectedness of
	solution sets in set
	optimization, vector equilibrium, or vector variational inequality under some
	generalized convexity/monotone assumptions. We plan to exploit special
	structures in geometry
	such as polyhedral or semi-algebraic 
	to discover new properties concerning contractibility or connectedness of solution sets
	of these problems as in
	\cite{huong2017connectedness,huong2016polynomial,hieu2019polynomialVVIS,hieu2020application}.


\begin{acknowledgements}
The authors are indebted to the anonymous referee for the insightful comments and valuable suggestions. The ERCIM Alain Bensoussan Fellowship Programme supported the research of Vu Trung Hieu.
\end{acknowledgements}



	\bibliographystyle{spmpsci_unsrt}
	\bibliography{references.bib}

\end{document}